\documentclass[review,onefignum,onetabnum]
{siamart171218}

\usepackage{lipsum}
\usepackage{placeins}
\usepackage{amsfonts}
\usepackage{graphicx}
\usepackage{epstopdf}
\usepackage{algorithmic}
\ifpdf
  \DeclareGraphicsExtensions{.eps,.pdf,.png,.jpg}
\else
  \DeclareGraphicsExtensions{.eps}
\fi

\newsiamremark{remark}{Remark}
\newsiamremark{hypothesis}{Hypothesis}
\crefname{hypothesis}{Hypothesis}{Hypotheses}
\newsiamthm{claim}{Claim}

\headers{Piecewise linear maps with heterogeneous chaos}{Y. Saiki, H. Takahasi and J. A. Yorke}

\title{Piecewise linear maps with heterogeneous chaos
}

\author{Yoshitaka Saiki\thanks{Graduate School of Business Administration, Hitotsubashi University, Tokyo 186-8601, Japan
(\email{yoshi.saiki@r.hit-u.ac.jp})}
\and Hiroki Takahasi\thanks{Keio Institute of Pure and Applied Sciences (KiPAS), Department of Mathematics, Keio University, Yokohama 223-8522, Japan (\email{hiroki@math.keio.ac.jp})}
\and James A Yorke\thanks{
Institute for Physical Science and Technology, Department of Mathematics and Department of Physics, University of Maryland  College Park, MD 20742, USA (\email{yorke@umd.edu})}
}

\usepackage{amsopn}

\newif\ifincludeJUNK
\includeJUNKtrue 
\newif\ifincludeXX
\includeXXtrue %

\newif\ifincludeOLD
\includeOLDtrue 
\includeOLDfalse 

\usepackage{xcolor}
\usepackage{mwe}
\usepackage{subfig}

\newcommand{\BF}{\boldmath}
\newcommand{\UNBF}{\unboldmath}

\newcommand{\allblack}{\color{black}{}}

\newcommand{\ie}{\textit{i.e.}}

\newcommand{\N}{{\mathbb{N}}}%

\newcommand{\mm}{\mathbb}

\newcommand{\A}{A}
\newcommand{\B}{B}
\newcommand{\C}{C}
\newcommand{\D}{D}

\newcommand{\mbf}{\mathbf} 

\newcommand{\avphi}{\langle\phi\rangle}

\newcommand{\cube}{{[0,1]^3}}

\newcommand{\HH}{H}

\definecolor{new}{rgb}{.3,.9,0}

\newcounter{arxiv}
\hypersetup{
  pdftitle=
  {Piecewise linear maps with heterogeneous chaos},
  pdfauthor={Yoshitaka Saiki, \and Hiroki Takahasi, \and James A Yorke
  }
}

\begin{document}

\maketitle

\begin{abstract}
Chaotic dynamics 
can be quite heterogeneous in the sense that in some regions the dynamics are unstable in more directions than in other regions.
When trajectories wander between these regions, the dynamics is complicated.
We say a chaotic invariant set is heterogeneous when arbitrarily close to each point of the set there are different periodic points with different numbers of unstable dimensions. We call such dynamics heterogeneous chaos (or hetero-chaos). While we believe it is common for 
physical systems to be hetero-chaotic, few explicit examples have been proved to be hetero-chaotic. Here we present two explicit dynamical systems that are particularly simple and tractable with computer. 
It will give more intuition as to how complex even simple systems can be. 
Our maps have one dense set of periodic points whose orbits are 1D unstable and another dense set of periodic points whose orbits are 2D unstable. Moreover, they are ergodic relative to the Lebesgue measure.
\end{abstract}

\section{Introduction}
Picture a chaotic trajectory in
a heterogeneous attractor:
there are a variety of regions $R_k$ 
where the local dynamics are expanding in $k$ directions and these occur 
for a variety of values of $k$. This picture may be representative of most high-dimensional chaotic attractors.
It suggests the existence of periodic orbits with different unstable dimensions.
Different periodic orbits can have different unstable dimensions, 
each of which type can be dense.
We call such maps or dynamics ``hetero-chaotic'', a term defined more precisely below.
Hetero-chaos may be more widespread than currently acknowledged.

Due to the difficulty in detecting such periodic orbits practically,
various approaches have been considered.
For example, 
the number of positive finite-time Lyapunov exponents were used 
to
detect the coexistence of periodic orbits of different unstable
dimensions in low-dimensional maps~
\cite{dawson_1996,dawson_1994,kostelich_1997}.
Periodic orbits of different unstable dimensions were used to discuss the occurrence of 
on-off intermittency~\cite{pereira_2007}.
For a continuous-time
atmospheric model, Gritsun~\cite{gritsun_2008,gritsun_2013}
found many periodic orbits with a wide
variety of unstable dimensions, all coexisting in the same system.
For the high dimensional Lorenz equation~\cite{lorenz_1996}, periodic orbits with different unstable dimensions were numerically detected in a single chaotic attractor~\cite{saiki_2018c}. 
These and other numerical results
have created among scientists, including applied mathematicians, a strong need for a simple model which qualitatively describes hetero-chaos.

Theoretical results on hetero-chaos came earlier than numerical ones. The existence of diffeomorphisms in which periodic orbits of different unstable dimensions are mixed was already known since the 60s 
\cite{abraham_1970,bonatti_1996,mane_1978,shub_1971,simon_1972}. These theoretical results are aimed at obtaining a global picture on dynamics of ``most'' diffeomorphisms of a compact Riemannian manifold,
and not aimed at understanding specific physical systems. To our knowledge, there is no rigorous proof that in any specific physical system hetero-chaos can actually occur.

Several simple models have been introduced and investigated which are relevant to hetero-chaos.
For a 2D coupled tent map \cite{pikovsky_1991}, Glendinning \cite{glendinning_2001} proved
the density of repellers in the attractor, while the density of saddles in the attractor is not known. 
The result of D\'iaz et al.~\cite{diaz_2014} 
suggests the existence of a hetero-chaos
in a certain H\'enon-like family.
Kostelich et al.~\cite{kostelich_1997} introduced a skew-product type analytic two dimensional map.
For this map with a certain parameter setting, Das and Yorke~\cite{das_2017a} proved the existence of a hetero-chaotic attractor
with a quasi-periodic orbit. In  \cite{saiki_2018c}, hetero-chaotic piecewise linear maps were announced without proofs. 

The goal of this paper is to verify the existence of  hetero-chaos in piecewise linear maps
as in \cite{saiki_2018c}.
We divide a space (a square $[0,1]^2$ or cube $[0,1]^3$) into several homogeneous regions with different expansion properties. 
Figure~\ref{fig:baker}(a) has only one region $R_1$ and Fig.~\ref{fig:baker}(b) has only one region $R_2$, while Figs.~\ref{fig:2D_HC_baker} and
\ref {fig:3H} have regions $R_1$ and $R_2$ where the dynamics are one- and two-dimensionally expanding, respectively.
To emphasize the meandering nature of the dynamics, we focus on periodic orbits that spend different amounts of time in the regions $R_k$ 
for different values of $k$. 
Periodic orbits are structures that can be found  numerically. 
The variability of their unstable dimensions 
is a signature of the heterogeneity of the system.

\begin{figure}
            \subfloat[]{\includegraphics[width=.49\textwidth]{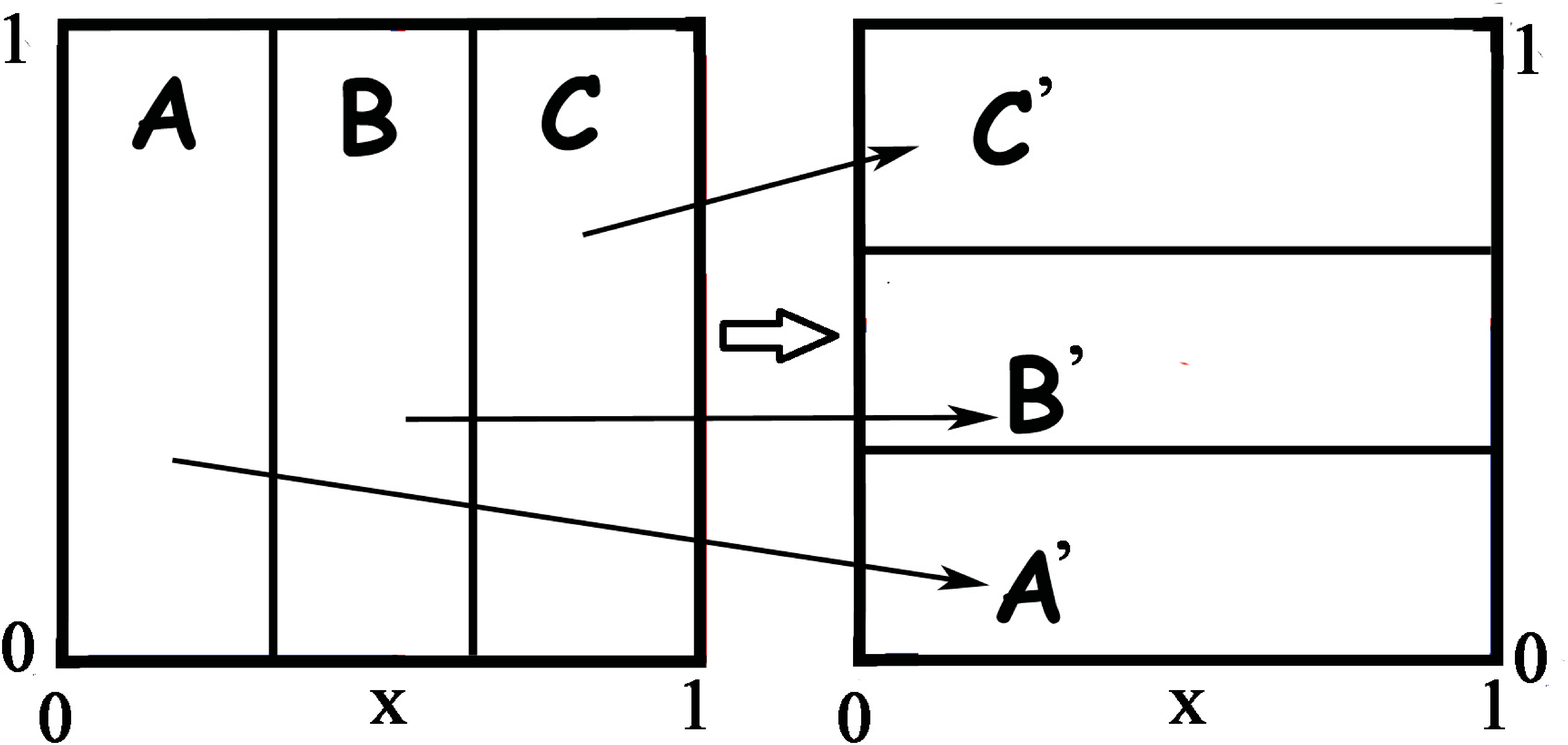}}
      \subfloat[]{\includegraphics[width=.49\textwidth]{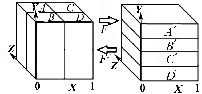}
      }
\caption{{\bf The homogeneous chaotic baker maps ``2D-baker'' and ``3D-baker''.}
The images of rectangles or boxes are denoted by primes~$\mbf {^\prime}$ so for example
$\A$ maps to $\A'$. 
{\bf (a):}
The 2D-baker map is defined by splitting 
${[0,1]^2}$ into 
three equal-sized vertical rectangles.
The square ${[0,1]^2}$ is mapped to itself, by
$x\mapsto \tau(x)$ in Eq.~\eqref{tau}.
Each vertical rectangle on the left maps to a different horizontal rectangle on the right, 
stretched horizontally and contracted vertically. 
{\bf (b):}
We define a 3D-baker map by slicing $\cube$ into four equal-sized breadbox-shaped boxes as shown.
The cube $\cube$ is mapped to itself, by
$x\mapsto 2x \bmod1$ and $z\mapsto 2z \bmod1$. 
}
\label{fig:baker}
\end{figure}

\subsection{ Hetero-chaos}\label{sec:defheterochaos}
Let $F\colon M\to M$ be a piecewise differentiable map.
 We say
a point $p\in M$ is a 
{\bf periodic point} of period $n$ 
if $F^n(p) = p.$ Let $p$ be a periodic point of period $n$
and assume that $F^n$ is differentiable at $p$.
The number of eigenvalues of the Jacobian matrix $DF^n(p)$ outside the unit disk $\{z\in\mathbb C\colon|z|\leq1\}$ counted with multiplicity is called the {\bf unstable dimension} of $p$. 
 If the unstable dimension of $p$ is equal to $k$, we say $p$ is {\bf $k$-unstable}.
We say $p$ is {\bf hyperbolic} if 
no eigenvalue of $DF^n(p)$ lies on the unit circle.

We say
an invariant  set $\Lambda\subset M$ (\ie, $F(\Lambda)=\Lambda$) is {\bf chaotic} if
\begin{itemize}
    \item[(i)]$\Lambda$ is an uncountable closed set.
\item[(ii)]the set of hyperbolic periodic points is dense in $\Lambda$, and 
\item[(iii)]$\Lambda$ has a dense trajectory, \ie, there exists $x_0\in \Lambda$ such that
the set\\
$(F^n(x_0))_{n\in \N}$ is dense in $\Lambda$, 
where $\N$ denotes the set of non-negative integers.
\end{itemize}
We say $\Lambda$ is {\bf hetero-chaotic} if (i), (iii) and the following hold:
\begin{itemize}
    \item[(ii$^\prime$)] for at least two values of $k$, the set of 
 $k$-unstable hyperbolic
 periodic points is dense in $\Lambda$. 
\end{itemize}
We say $F$ is hetero-chaotic on $\Lambda$. 
In the case $\Lambda=M$ we simply say $F$ is hetero-chaotic,
or has heterogeneous chaos, or more briefly {\bf hetero-chaos} or even {\bf HC}.

\begin{figure}
\centering
 \includegraphics[width=.65\textwidth]{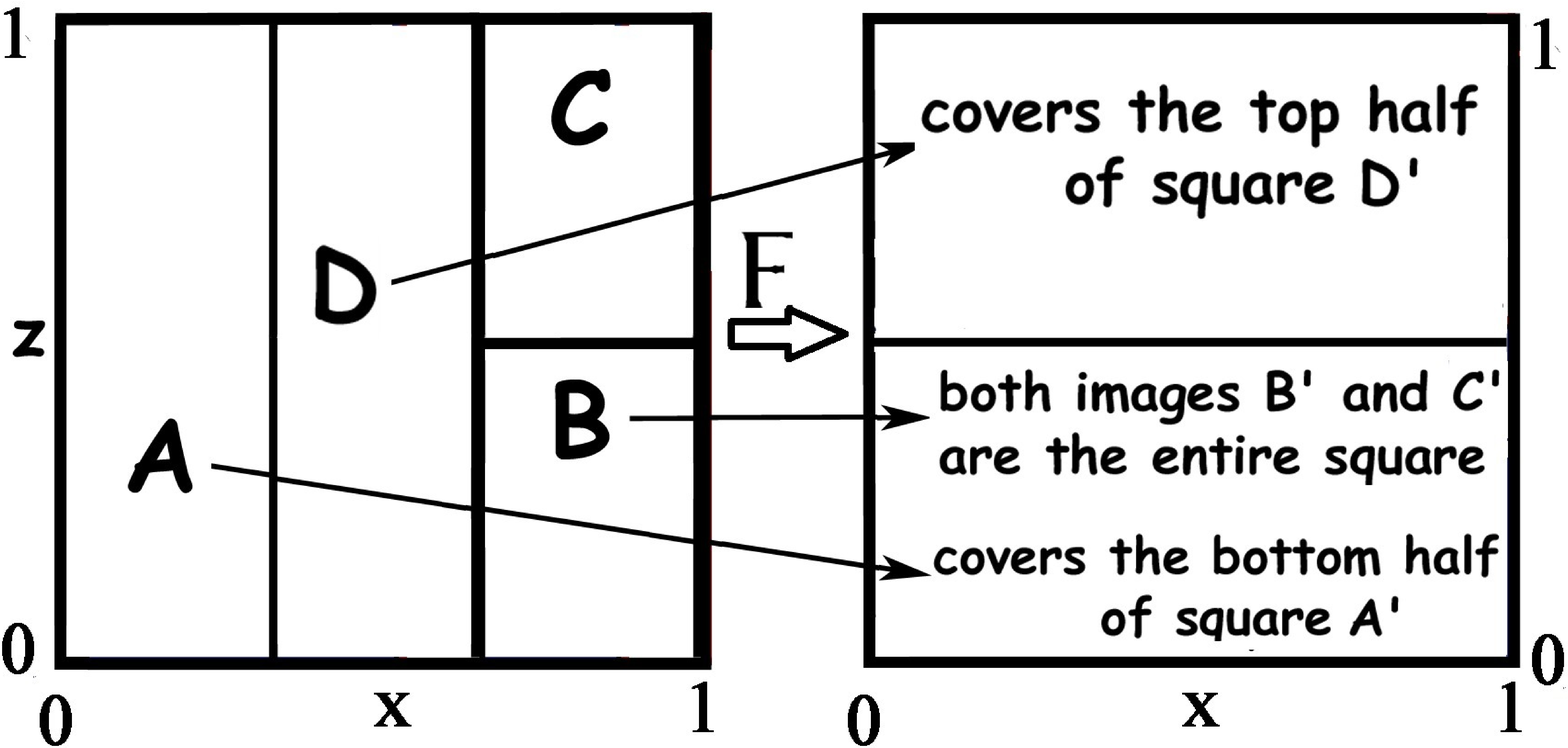}
\caption{
 {\bf 
The 2D hetero-chaos baker map}. 
The square $[0,1]^2$ is divided into four rectangles $A = \left[0,\frac{1}{3}\right)\times\left[0,1\right],$
$B =\left [\frac{2}{3},1\right]\times\left[0,\frac{1}{2}\right),$
$C = \left[\frac{2}{3},1\right]\times\left[\frac{1}{2},1\right],$
$D =\left[\frac{1}{3},\frac{2}{3}\right)
\times\left[0,1\right].$
The 2D-HC map $F$ expands each rectangle horizontally to full width as shown. The region $R_1 = A\cup D$ is contracted vertically. The region 
$R_2= B\cup C$ is expanded in both coordinates so that each of the images of $B$ and $C$ covers $[0,1]^2$. 
Hence $R_1$ and $R_2$ are regions of one- and two-dimensional instability. While the maps in Fig.~\ref{fig:baker} are one-to-one almost everywhere, this map is not.}
\label{fig:2D_HC_baker}
\end{figure}

\subsection{Baker maps}
We begin with 
the well-known ``baker map'' (Fig.~\ref{fig:baker}(a)). 
We refer to it as the 2D-baker map. 
It was introduced by Seidel \cite{seidel_1933}, where the square is divided into $q$ equal vertical strips. Seidel used $q=10$. We use $q=3$, 
though $q=2$ is most common in the literature.
Each strip is mapped by a map $F$ to a horizontal strip by squeezing it vertically by the factor $\frac{1}{3}$ and stretching it horizontally by the factor $3$. The resulting horizontal strips are laid out covering the square $[0,1]^2$.
There, all periodic points are $1$-unstable.
We show a three-dimensional version that we call a ``3D-baker map'' in Fig.~\ref{fig:baker}(b). 
 All periodic points of the 3D-baker map are 2-unstable. 
The baker maps here are area or volume preserving and are one-to-one almost everywhere. 
When all periodic points have the same unstable dimension, we call the chaos {\bf homogeneous}. 
\subsection{3D hetero-chaos baker map}
We modify the two baker maps in Fig.~\ref{fig:baker}  to hetero-chaotic maps, which we 
 call the ``2D-HC'' and ``3D-HC'' maps respectively, as
in Figs.~\ref{fig:2D_HC_baker} and ~\ref{fig:3H}. 
We consider these as prototypes for understanding attractors with far higher dimensions.

\begin{figure}
\centering
\subfloat[]{\includegraphics[height=0.33\textwidth,width=0.9\textwidth]{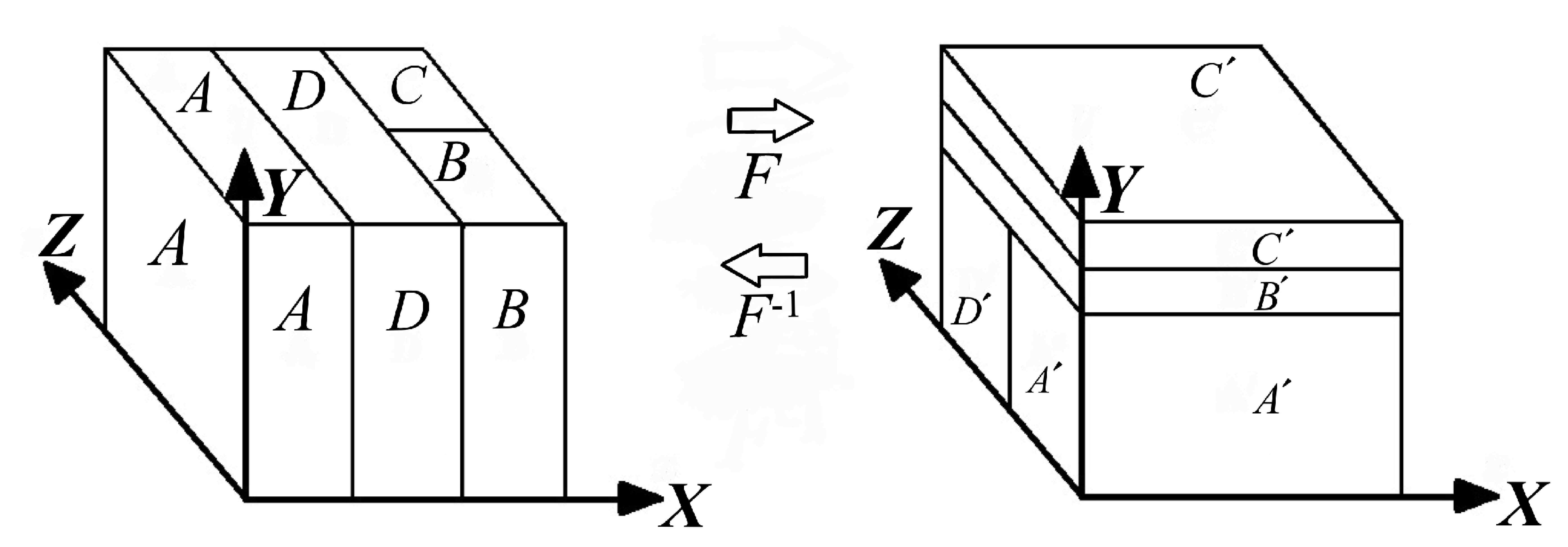}
}\\
\begin{minipage}{.5\linewidth}
\centering
\subfloat[]{\includegraphics[height=0.6\textwidth,width=0.65\textwidth]{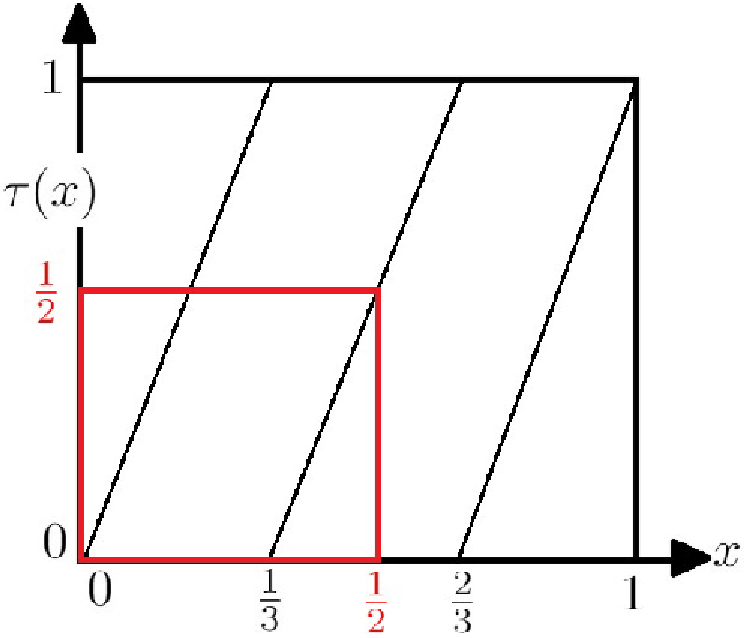}}
\end{minipage}%
\begin{minipage}{.5\linewidth}
\centering
\subfloat[]{\includegraphics[height=0.63\textwidth,width=0.68\textwidth]{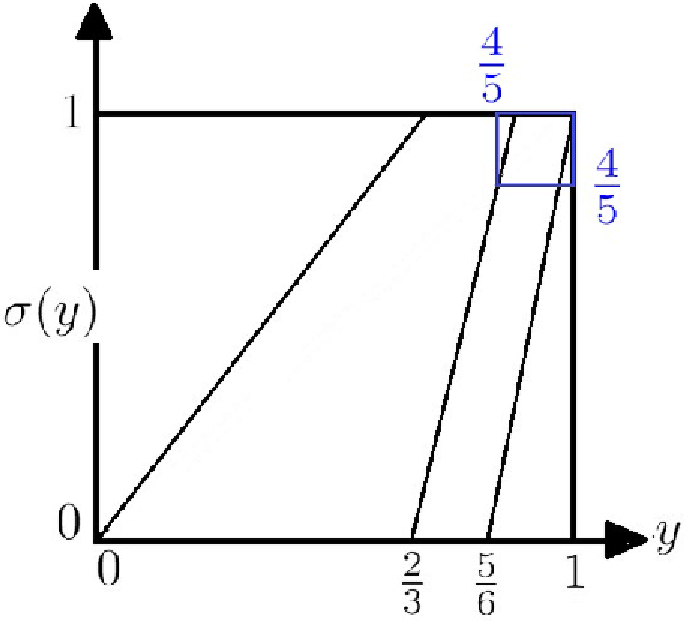}
}
\end{minipage}\par\medskip
\caption{
 {\bf 
 (a) The 3D hetero-chaos baker map.}
 Here the $XZ$ plane plays the role of $XZ$ in Fig.~\ref{fig:2D_HC_baker} and the $Y$ coordinate has been added.
Here $\cube$ is partitioned into four regions $\A$, $\B$, $\C$, $\D$ 
and each of them is mapped into a region of the same volume. 
The sets $\A$ and $\D$ each have volume $\frac{1}{3}$ and expand in the $X$ direction and 
contracting in the $Y$ and $Z$ directions.
The sets $\B$ and $\C$ both have volume $\frac{1}{6}$ and expand in $X$ and $Z$ directions and contracting in the $Y$ direction. 
{\bf (b) The graph of $\tau(x)$.}
{\bf (c) The graph of $\sigma(y)$.}
The first branch 
has slope 
$\frac{3}{2}$ while the other two branches have slope 
$6$. 
The red and blue boxes are enlarged in Fig.~\ref{fig:index-maps}. 
}\label{fig:3H}
\end{figure}

The 3D-HC map $F$ is defined as follows. Define $\tau\colon[0,1]\to[0,1]$ by
\begin{equation}
\label{tau}
    \tau(x)=\begin{cases}3x \mod 1&\text{ for }0\leq x<1,\\
    1&\text{ for }x=1.
\end{cases}
\end{equation}
Then
\begin{eqnarray}
  F((x,y,z))=
  \begin{cases}
    \left(\tau(x),~~~~~\frac{2}{3}y,~~~~~~~\frac{1}{2}z\right)&\text{on}~A \\
    \left(\tau(x),\frac{2}{3}+\frac{1}{6}y,~~~~~~~2z\right)&\text{on}~B \\
   \left (\tau(x),\frac{5}{6}+\frac{1}{6}y,-1+2z\right)&\text{on}~C \\
    \left(\tau(x),~~~~~\frac{2}{3}y,~~\frac{1}{2}+\frac{1}{2}z\right)&\text{on}~D,
  \end{cases}
  \label{eq:3D-HC}
\end{eqnarray}
where 
\begin{eqnarray*}
A =& \left[0,\frac{1}{3}\right)\times
\left[0,1\right]\times\left[0,1\right];\\
B =&\left [\frac{2}{3},1\right]\times
\left[0,1\right]\times\left[0,\frac{1}{2}\right);\\
C =& \left[\frac{2}{3},1\right]\times
\left[0,1\right]\times\left[\frac{1}{2},1\right];\\
D =&\left[\frac{1}{3},\frac{2}{3}\right)
\times\left[0,1\right]\times\left[0,1\right].
\end{eqnarray*}
{Note that the map $F$ is differentiable
and one-to-one except 
on the points in the boundaries of $A, B, C, D$ where it is discontinuous and at most three-to-one. With a slight abuse of notation,
we fix a piecewise linear map $F^{-1}$ on $[0,1]^3$
satisfying $F\circ F^{-1}(x,y,z)=(x,y,z)$
on $F([0,1]^3)$.
For $n\geq1$, let $F^{-n}$ denote the $n$-composition of $F^{-1}$.}
The map $F^{-1}$ has many of the features of $F$.
 The $Y$ coordinate $y'$ of $ 
F^{-1}(x,y,z)$ depends only on $y$, 
so we define 
$\sigma\colon[0,1]\to[0,1]$ by $y'=\sigma(y)$. See Figs.~\ref{fig:3H}(c) and \ref{fig:index-maps}(b).

Let $F$ be the 3D-HC map and
 let 
$\pi\colon[0,1]^3\to[0,1]^2$ 
be the projection $\pi(x,y,z)=(x,z)$.
We define the 2D-HC map by $$F_{2\text{D}}(x,z)=\pi(F(x,0,z)).$$ 
Note that $\pi\circ F=F_{2\text{D}} \circ \pi$. We will omit the subscript and write $F$ to denote both the 2D-HC map and the 3D-HC map.

\subsection{Heterogeneous chaos}
Figure~\ref{fig:per} shows parts of sets of hyperbolic periodic points numerically detected from the 2D-HC
map. It suggests that the set of 1-unstable hyperbolic periodic points is 
dense, and the set of 2-unstable hyperbolic periodic points is also dense.
We prove this as a main result of this paper.
 
\allblack
\begin{theorem}
\label{thm:main}
The 2D-HC and 3D-HC maps are hetero-chaotic.
\end{theorem}

 A well-known argument for showing the density of periodic points relies on the use of a Markov partition \cite{kitchens_1998}. However, hetero-chaotic systems never admit Markov partitions.
The lack of a Markov partition for our maps is due to the fact that the $Z$ direction is neither uniformly contracting or expanding. For example, the 2D-HC map has a segment
 $\{\frac{1}{4}\}\times(0,1)$ of periodic points of period $2$. 
 The corresponding segment of periodic points for the 3D-HC map is $\{(\frac{1}{4},\frac{15}{16})\}\times(0,1)$. 
No subshift of finite type on finitely many  symbols can model our maps since such a subshift has only countably many periodic points.

In Secs.~\ref{sec:periodic-points} and \ref{sec:ergodicity} we prove Thm.~\ref{thm:main}. Our idea is to recover geometric features of a Markov partition which lead to the heterochaos.
A key ingredient is {\it a brick} that we introduce in 
Sec.~\ref{sec:periodic-points}.
 We construct special bricks carefully analyzing the expansion rates in the $Z$ direction, and 
 establish one part of Thm.~\ref{thm:main} on periodic points of the 3D-HC map.
This geometric construction 
can be adapted to showing the existence of a dense trajectory in $\cube$, (Thm.~\ref{thm:transitivity}), thereby 
 establishing the heterochaos for the 3D-HC map.
  Since the projection $\pi$ maps each dense trajectory for the 3D-HC map to that for the 2D-HC map, and each periodic point for the 3D-HC map to that for the 2D-HC map preserving the unstable directions, the heterochaos for the 2D-HC map follows. 
 
 \begin{figure}
\centering
\begin{minipage}{.5\linewidth}
\centering
\subfloat[]{\includegraphics[width=.6\textwidth]{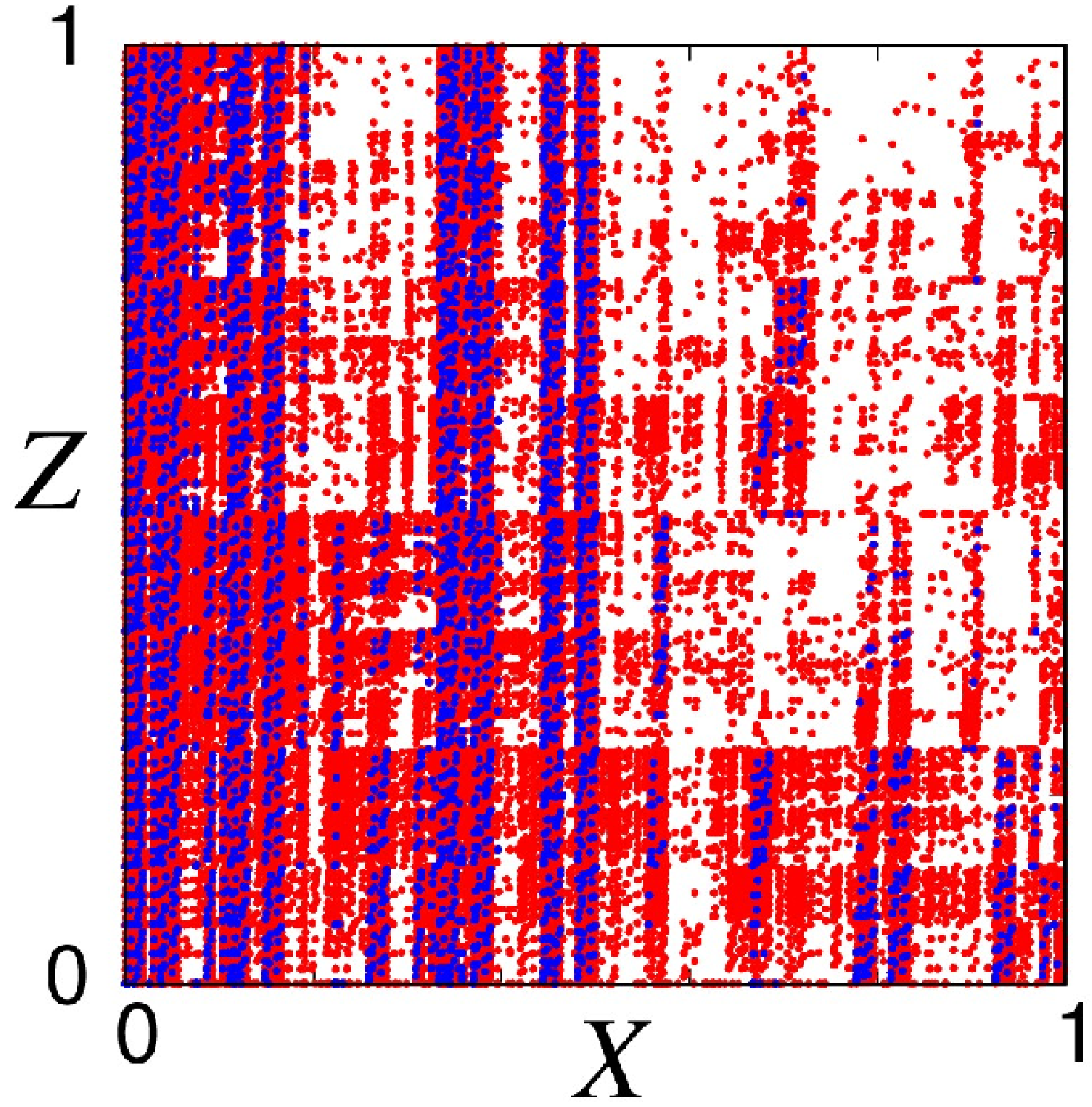}}
\end{minipage}%
\begin{minipage}{.5\linewidth}
\centering
\subfloat[]{\includegraphics[width=.6\textwidth]{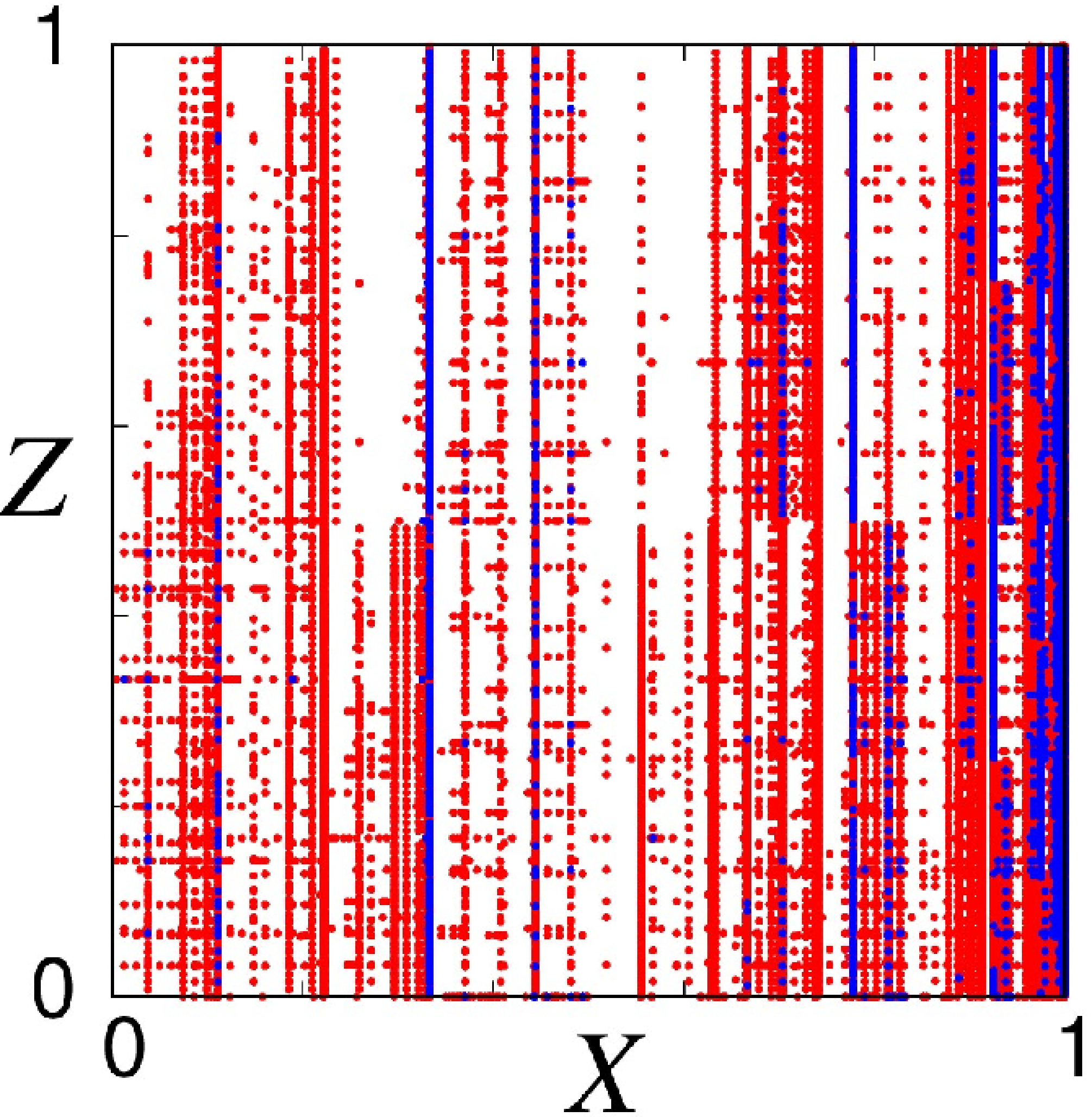}}
\end{minipage}\par\medskip
\caption{{\bf Sets of {$k$-unstable hyperbolic} 
periodic points of the 2D-HC map.} {\bf (a):} 
{$k=1$} (periods 2-10~(blue), periods 11-13~(red))
 {\bf (b):} 
{$k=2$}
(periods 2-10~(blue), periods {11}-13~(red)). 
}\label{fig:per}
\end{figure}
\subsection{Ergodicity}
The 2D-HC map preserves area and the 3D-HC map preserves volume.
 It is natural to ask if they are ergodic. We say $F$ is {\bf ergodic}
 if for each $\mathbb R$-valued continuous function $\phi$
 on the space,
 the trajectory average
\begin{equation}\label{average}
\avphi(p)= \lim_{n\to\infty} \frac{1}{n}\sum_{i=0}^{n-1} \phi(F^i(p))
\end{equation}
 exists and coincides for almost every initial point $p$.
 
 \begin{theorem}[Ergodicity]\label{thm:ergodic0}
 The 2D-HC and 3D-HC maps are ergodic.
\end{theorem}

In Sec.~\ref{sec:ergodicity} we prove Thm.~\ref{thm:ergodic0}.
 A key observation for the 3D-HC map is that a set of the form
 $\{x\}\times[0,1]^2$, which we will call {\it a leaf at $x$}, is mapped into the leaf at $\tau(x)$.
 The ergodicity would follow immediately from the ergodicity of $\tau$
 \cite[Theorem~1.11]{walters_1982}, if each leaf shrank to a point under forward iteration.
 In fact some leaves do not shrink, notably those containing 2-unstable periodic points. 
   We show that almost every leaf shrinks to a point under forward iteration (Prop.~\ref{prop:leaf}),
 and use this result to establish the ergodicity of the 3D-HC map.
  Using the projection $\pi$ one can quickly deduce the ergodicity of the 2D-HC map.

In the same way
we can show 
that $F\times F$ is ergodic, and hence
$F$ is weak mixing
\cite[Theorem~1.24]{walters_1982}.
The following is an immediate consequence of this.
We say a pair $p_1,p_2$ of points in the space is a {\bf  scrambled pair} \cite{Li_1975} if
\[\liminf_{n\to\infty}{\rm dist}(F^n(p_1),F^n(p_2))=0\quad\text{and}\quad\limsup_{n\to\infty}{\rm dist}(F^n(p_1),F^n(p_2))={\rm diam},\]
where ${\rm dist}$ denotes the Euclidean distance and ${\rm diam}$ denotes the diameter of the space.
\begin{theorem}\label{li-yorke}
Let $F$ be the 2D-HC or the 3D-HC map.
Then almost every pair is a scrambled pair.
\end{theorem}


\subsection{Homogeneous regions and index sets} 
For the 3D-HC map $F$, we provide some understanding of the dynamics between different regions of $[0,1]^3$.
Set \[R_1=\A\cup\D\quad\text{and}\quad R_2=\B\cup\C,\] and define
\[H_1=\bigcap_{n=-\infty}^\infty F^{-n}(R_1)\quad\text{and}\quad
H_2=\bigcap_{n=-\infty}^\infty F^{-n}(R_2).\]
The unstable and stable dimensions of all the points in $H_1$ are $1$ and $2$, and 
 the unstable and stable dimensions of all the points in $H_2$ are $2$ and $1$. These two sets, called {\bf index sets} \cite{saiki_2018c}, are the source of heterogeneous chaos.
   See Fig.~\ref{fig:index}.

   \begin{figure}
\centering
  \subfloat[]{
    \includegraphics[width=.365\textwidth,height=.31\textwidth]{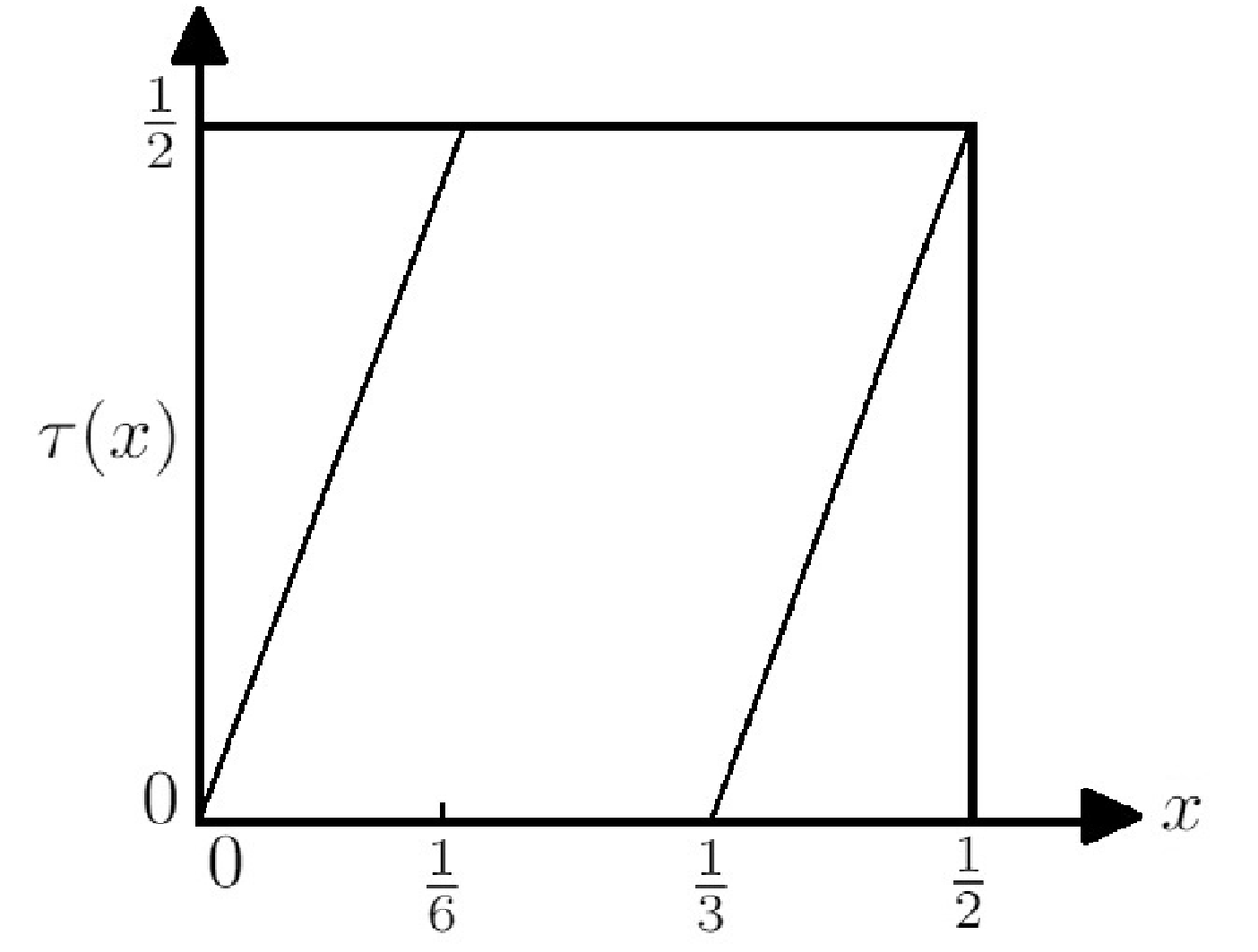}}
    \subfloat[]{
    \includegraphics[width=.385\textwidth,height=.305\textwidth]{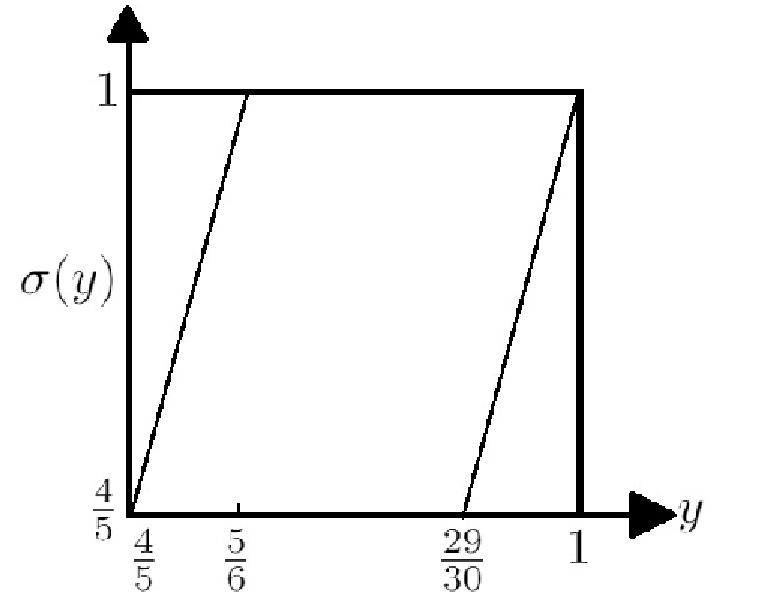}}
\caption{ {\bf Detail of Fig.~\ref{fig:3H} showing  parts of the domains of $\tau(x)$ and $\sigma(y).$} 
{\bf (a):} This is the restriction of $\tau$ to {$[0,\frac{1}{2}]$}, 
the smallest interval containing all the $x$ values of points 
in $\bigcap_{n=0}^{\infty} F^{-n}(R_1).$
The endpoints $0$ and $\frac{1}{2}$ are fixed points of $\tau$. 
The slope is $3$.
{\bf (b):} This is the restriction of $\sigma$ to {$[\frac{4}{5},1]$},
the smallest interval containing all the $y$ values of points in
$\bigcap_{n=0}^{\infty} F^{n}(R_2).$ 
The endpoints $\frac{4}{5}$ and $1$ are fixed points of $\sigma$. 
The slope is $6$.
}
\label{fig:index-maps}
\end{figure}

The {\bf heteroclinic sets} $\HH_{1,2}$ and $\HH_{2,1}$
are the following:
\begin{align*}
\HH_{2,1}&=
\bigcup_{m=0}^\infty F^{-m}\left(\bigcap_{n=0}^\infty F^{-n}(R_1)\right)\cap
\bigcup_{m=0}^{\infty} F^{m}\left(\bigcap_{n=0}^{\infty} F^{n}(R_2)\right),\\
\HH_{1,2}&=
\bigcup_{m=0}^\infty F^{-m}\left(\bigcap_{n=0}^\infty F^{-n}(R_2)\right)\cap
\bigcup_{m=0}^{\infty} F^{m}\left(\bigcap_{n=0}^{\infty} F^{n}(R_1)\right).
\end{align*}
These sets can be described in more detail as follows: $p\in\HH_{2,1}$
if and only if $p$ is in the stable  set of a point in $\HH_1$ and is in the unstable 
 set of a point in $\HH_2$, i.e.,
  there are integers $n^+(p)\ge0$ and $n^-(p)\le 0$
 such that $F^n(p)\in R_1$
for all $n \ge n^+(p)$ and $F^n(p)\in R_2$ for all $n \le n^-(p)$. The set $\HH_{1,2}$ can be described analogously, switching the subscripts $1$ and $2$.
Two subsets $\HH_{2,1}^*$ and $\HH_{1,2}^*$ defined by
\begin{align*}
\HH_{2,1}^*&=
\bigcap_{n=0}^\infty F^{-n}(R_1)\cap\bigcap_{n=1}^\infty F^{n}(R_2)\subset \HH_{2,1},\\
\HH_{1,2}^*&=
\bigcap_{n=0}^\infty F^{-n}(R_2)\cap\bigcap_{n=1}^\infty  F^{n}(R_1)\subset \HH_{1,2},
\end{align*}
are easy to visualize as shown in 
 Fig.~\ref{fig:index}. For example $\HH_{2,1}^*$ is the set of $p\in \HH_{2,1}$ for which
$n^+(p) =0$ and  $n^-(p) = -1$.

\begin{figure}
\centering
\begin{minipage}{.65\linewidth}
\centering
\subfloat[]{\includegraphics[width=.65\textwidth]{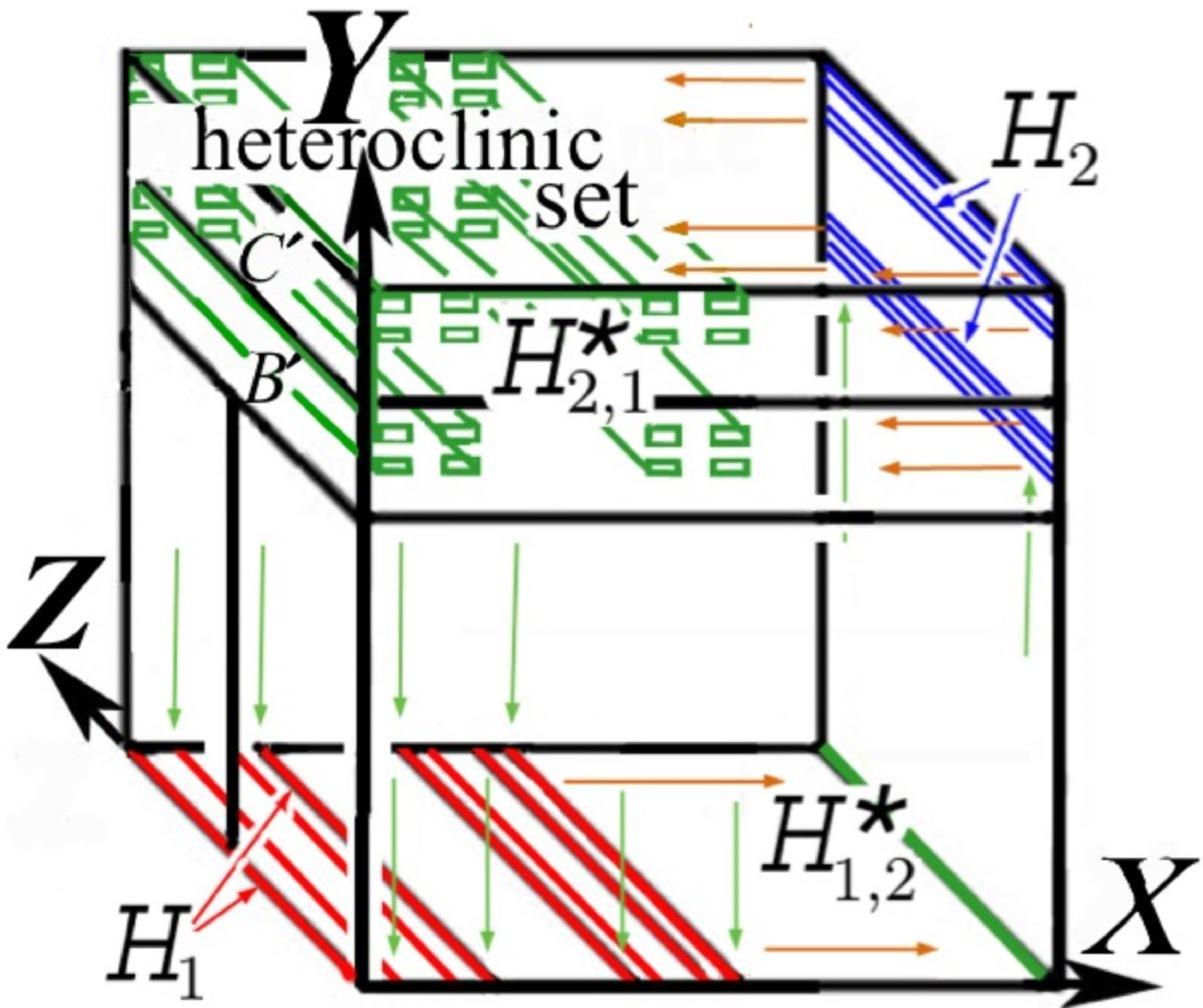}}
\end{minipage}%
\begin{minipage}{.35\linewidth}
\centering
\subfloat[]{\includegraphics[width=.65\textwidth,height=1.1\textwidth]{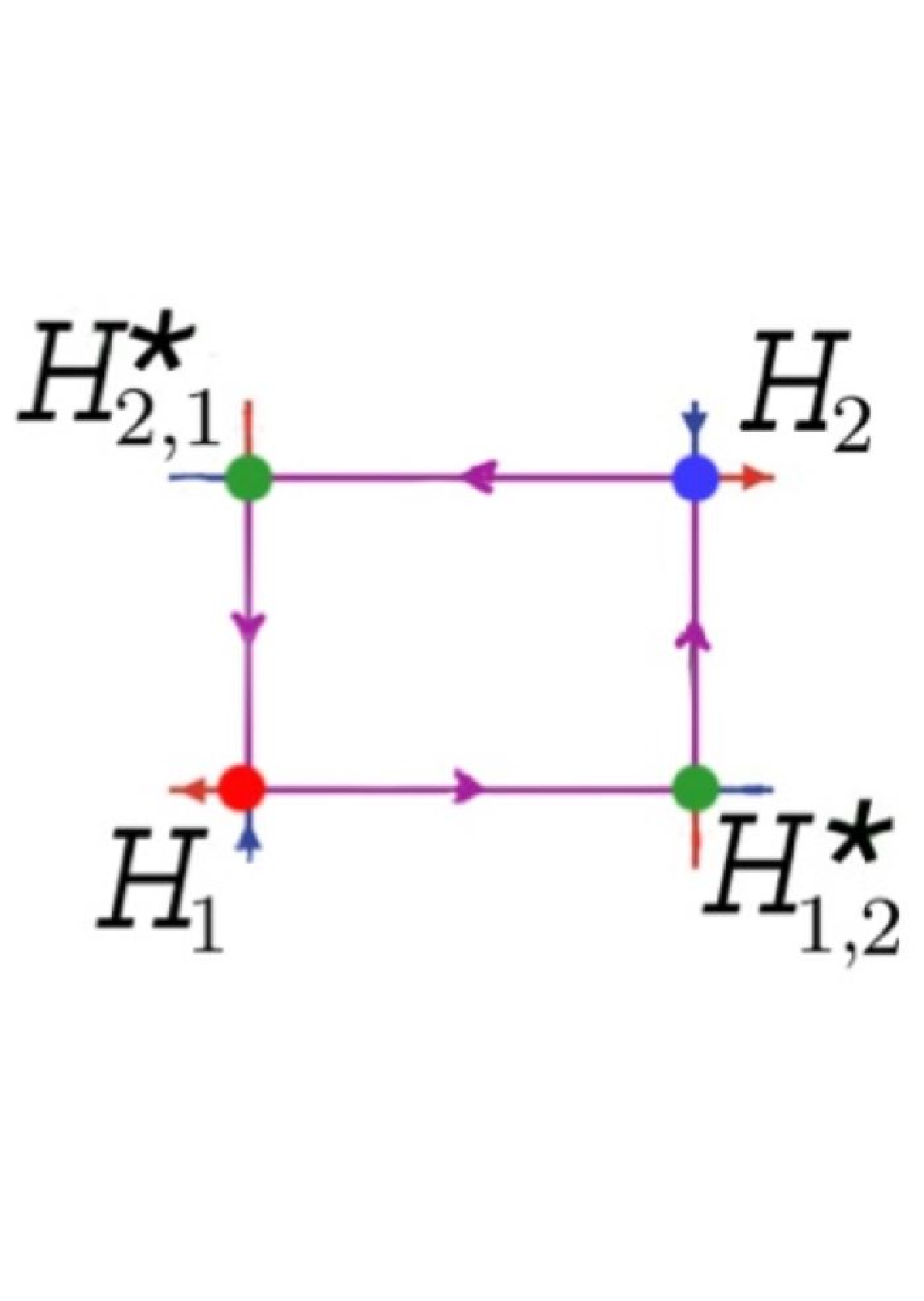}}
\end{minipage}\par\medskip
\caption{
{\bf The index sets for the 3D-HC map.}
{\bf (a):}
The 1D unstable index set $\HH_1$ shown in red lies in the plane $y=0$.
The 2D unstable index set $\HH_2$ shown in blue lies in the plane $x=1$.
All are shown in the cube partitioned according to the symbol sets of $F^{-1}$.
The vertical (green) arrows show the stable directions of $\HH_1$ and $\HH_2$; the horizontal (red) arrows show their unstable directions. See the text for the heteroclinic sets 
$\HH^{*}_{1,2}(\subset\HH_{1,2})$ (a straight line segment), $\HH^{*}_{2,1}(\subset\HH_{2,1})$ (the product of two crudely drawn Cantor sets with a straight line segment parallel to the $Z$-axis).
See also Fig.~\ref{fig:index-maps}.
{\bf (b):}
A symbolic representation of the mutual intersections of the stable and unstable sets of $\HH_{1}$ and $\HH_{2}$.
}\label{fig:index}
\end{figure}

\ifincludeXX

 \subsection{A modified 3D-HC map without a dominant expanding direction}
 \label{sec:nondominant}

In choosing the definition of the 
3D-HC map, we have for simplicity 
chosen to have a dominant expanding 
direction, the $X$ direction.
An alternative formulation of Fig.~\ref{fig:3H} replaces the 
region $\B$ with regions 
with regions $B_i$, $1\leq i\leq k$ stacked vertically. 
See Fig.~\ref{fig:3Hgen}.
In the case $k>4$, the locally dominant expanding direction is $Z$ for $\frac{2}{3}<x<1$ 
and is $X$ for $0<x<\frac{2}{3}$. 
Our results hold for these modified maps. Our proof of ergodicity requires 
$\frac{1}{2}^{2/3}k^{1/3}\neq1$ which is true if $k\neq4$.
In the case $k<4$, one can show the ergodicity of $F^{-1}$ using the argument in this paper. The ergodicity of $F^{-1}$ implies that for $F$.

\begin{figure}
\centering
\includegraphics[width=.93\textwidth,height=.35\textwidth]{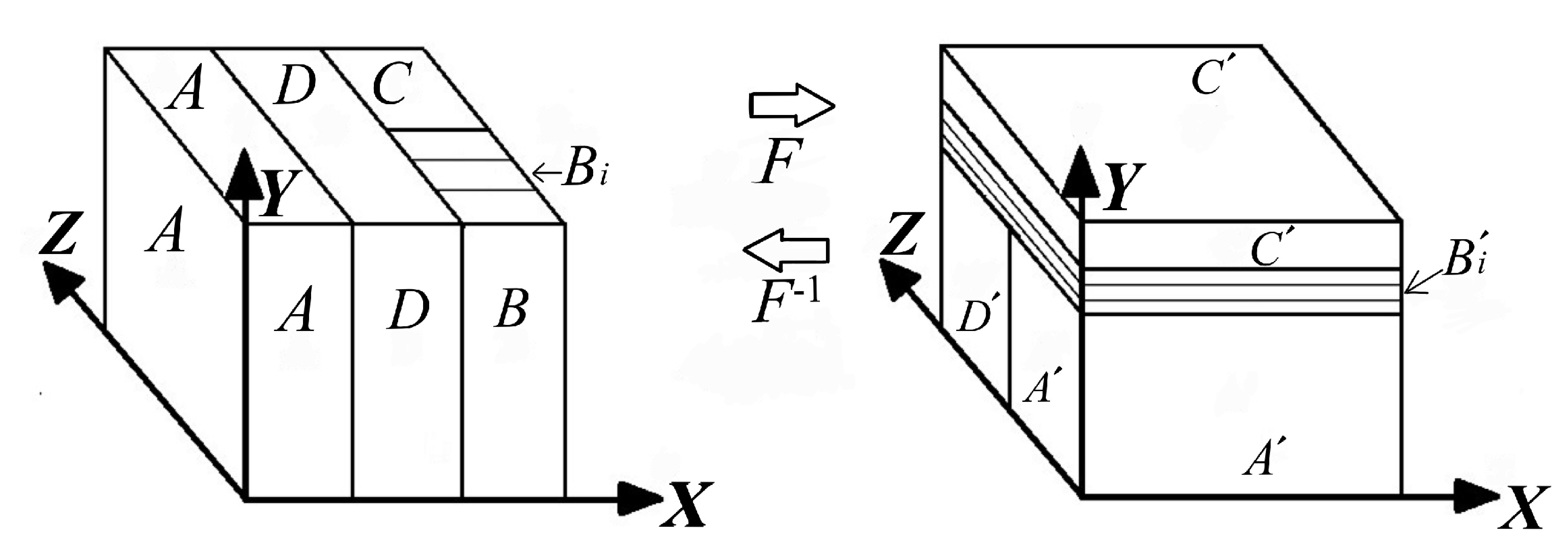}
\caption{{\bf A modified 3D-HC map without a dominant expanding direction.}
A modified 2D-HC map can be similarly defined.  
}\label{fig:3Hgen}
\end{figure}

\section{Establishing heterogeneous chaos}
\label{sec:periodic-points}
This section is devoted to the proof of Thm.~\ref{thm:main}.
After preliminaries in Sec.~\ref{box-s}, 
we introduce {\it $(j,k)$-bricks} in Sec.~\ref{brick},
as a tool to prove the existence of
a hyperbolic periodic point. 
In Sec.~\ref{brick-baker} we display $(j,k)$-bricks 
for the 3D-baker map, and use them to show the density of hyperbolic periodic points.
This argument is a paraphrase of the standard one by way of a Markov partition, and does not work for the 3D-HC map.
To obtain an expansion in the $Z$ direction required in the definition
of a $(j,k)$-brick,
 we introduce biased points in Sec.~\ref{sect-bias},
and from them construct $(j,k)$-bricks in Sec.~\ref{brick-hc}.
In Sec.~\ref{dense-per} we establish one part of Thm.~\ref{thm:main} on hyperbolic periodic points of the 3D-HC map.
In Sec.~\ref{sec:transitivity} we show the existence of a dense trajectory and complete the proof of Thm.~\ref{thm:main}.

\subsection{ Boxes and symbol sets}\label{box-s}
We say a set 
$Q\subset[0,1]^3$ 
is a {\bf box} if it is the Cartesian product of three non-empty intervals.
Two kinds of boxes that will repeatedly occur deserve special attention. In the spirit of the baker of the baker map, we name them after two boxes a baker might use. 
\allblack 
We say a box $Q$ is a {\bf breadbox} if precisely one coordinate has length $1$.
For example, if its $X$ coordinate has length $1$, we can call $Q$ an $X$ breadbox.
If $Q$ has precisely two coordinates 
of length 1, we call it a {\bf pizzabox}. For example,
if its $X$ and $Z$ coordinates have length $1$, we call $Q$
an $XZ$ pizzabox. For the 3D-HC map $F$,
$A$ and $D$ are $YZ$ pizzaboxes 
and 
$A'=F(A)$ and $D'=F(D)$ are $X$ breadboxes.
Similarly, 
$B$ and $C$ are 
$Y$ breadboxes
and  $B'=F(B)$ and $C'=F(C)$
are $XZ$ pizzaboxes.

For the 3D-baker or the 3D-HC map $F$, we refer to the boxes $\A,\B,\C,\D$ as {\bf symbol sets} for $F$.
On each symbol set, $F$ 
has the form
\begin{equation}
p\mapsto (c_X(p) + d_X(p)x,c_Y(p) + d_Y(p) y, c_Z(p) + d_Z(p) z),
\label{mapsto:affine}
\end{equation} 
 where $p=(x,y,z)$ and the three $c$'s and $d$'s depend on the symbol set
 containing $p$.
Since the Jacobian matrix of such a map is diagonal, we refer to a map as in \eqref{mapsto:affine} as a {\bf linear diagonal map}.
The inverse map $F^{-1}$ is 
a linear diagonal map on each of the boxes $F(S),$ $S\in\{A,B,C,D\}$ which are symbol sets for $F^{-1}$.

\begin{figure}
\centering
\includegraphics[width=.40\textwidth]{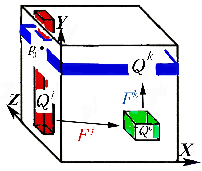}
\caption{
{\bf 
Construction of a brick containing a periodic point.}
This figure shows a 
key concept for the proof of the existence of a 2-unstable periodic point in $(0,1)^3$ for the 3D-baker or the 3D-HC map.
The box $Q^0\subset(0,1)^3$ is chosen to be an ``interior $(j,k)$-brick'', thereby guaranteeing that there is a periodic point $p_0\in Q^{-j}\cap Q^k$ of period $j+k$.
It follows that $F^j(p_0)\in Q^0$ is a periodic point of period $j+k$.
}
\label{fig:brick} 
\end{figure}

\subsection{Construction of hyperbolic periodic points from bricks}\label{brick}
Let $F$ be the 3D-baker or the 3D-HC map.
The $X$ direction is unstable and the $Y$ direction is stable, 
in that $d_X=3>1$ and $d_Y<1$ on every symbol set.
If $p$ is a periodic point of period $n$,
the number $\chi_Z(p)= \prod_{i=0}^{n-1}~{d_Z}(F^i(p))$
determines 
the type of $p$:
it is 2-unstable if $\chi_Z(p)>1$, and 
1-unstable if $\chi_Z(p)<1$.

We introduce a key ingredient for constructing a $2$-unstable periodic point of $F$.
Let $j,k\in \N$ and assume $k\geq1$.
For a box $Q^0$ and an integer $-j\leq i\leq  k$ we write 
     $\BF Q^i\UNBF= F^i(Q^0),$
     $Q^i={Q^{i}}_{X}\times{Q^{i}}_{Y}\times{Q^{i}}_{Z}$.
We say $Q^0$
is a $ \boldsymbol{(j,k)}$-{\bf brick}
if 
\begin{itemize}
\item[$\circ$] for each $-j\leq i\leq k-1$,
$Q^i$ is a subset of some 
symbol set for $F$; 
\item[$\circ$]
$Q^k$ 
is an $XZ$ pizzabox; 
\item[$\circ$] $Q^{-j}$ 
is a $Y$ breadbox.
\end{itemize}
 We say a $(j,k)$-brick $Q^0$ is
an {\bf interior} $ \boldsymbol{(j,k)}$-{\bf brick}, or simply an {\bf interior brick}
if $\overline{{Q^{-j}}_X},\overline{{Q^{-j}}_Z}\subset(0,1)$ and 
$\overline{{Q^{k}}_Y}\subset(0,1)$.
%

The notion of an interior brick is motivated by the fixed point theorem of a contraction mapping: 
let $J$ be a bounded
interval and $f\colon J\to J$ be a map of the form
 $f(x)= c+dx$, $c,d\in\mathbb R$ 
 such that 
 $\overline{f(J)}$ is contained in the interior of $J$.
 Then 
 there is a unique fixed point of $f$ in the interior of $J$.
The following proposition extends this idea to the 3D-HC map,
see Fig.~\ref{fig:brick}.
The proof is standard but is included for completeness. 

\begin{proposition}\label{prop:periodicpoint}
Let $F$ be either the 3D-baker map or 
the 3D-HC map.
Let $j\ge0,k\geq1$ and let $Q^0$ be an interior $(j,k)$-brick.
Then the interior of $Q^0$ contains a $2$-unstable 
{hyperbolic} periodic point of period $j+k$. 
\end{proposition}
\begin{proof}
Since $Q^0$ is 
a $(j,k)$-brick, 
$F^{j+k}|_{Q^{-j}}\colon Q^{-j}\to Q^{k}$ is a linear diagonal map as in \eqref{mapsto:affine} where
$d_X,d_Z >1$ and $d_Y <1.$
Since $Q^{0}$ is an interior $(j,k)$-brick,
$\overline{{Q^{-j}}_{X}}\subset(0,1)$ and the $X$ coordinate of $F^{j+k}|_{Q^{-j}}$
maps ${Q^{-j}}_{X}$ linearly 
onto $[0,1)$,
so it has unique fixed point in $(0,1)$.
The same is true for the $Z$ coordinate.
The $Y$ coordinate of $F^{j+k}|_{Q^{-j}}$ maps ${Q^{-j}}_Y$ linearly onto an interval whose closure is contained in $(0,1)$.
In each case there is a unique fixed point, whose coordinates we denote as
$x_{-j}, y_{-j},z_{-j}$, each in the interiors of 
$ {Q^{-j}}_{X}$, ${Q^{-j}}_{Y}$, ${Q^{-j}}_{Z}$ respectively.
Then $p= (x_{-j},y_{-j},z_{-j})$ is in the interior of $ Q^{-j}$ and is a 
fixed point of $F^{j+k}$.
Then $F^{j}(p)\in Q^0$ 
is a fixed point for $F^{j+k}$.
Since $d_X,d_Z >1$ and $d_Y <1,$ $F^{j}(p)$
 is a 2-unstable hyperbolic periodic point of period $j+k$.
\end{proof}

\subsection{Bricks for the 3D-baker map}\label{brick-baker}
Let $j,k\in\N$, $k\geq1$.
All $(j,k)$-bricks for the 3D-baker map $F$ are of the form
\[
Q_{j,k}(a,b,c)=
\left(\frac{a}{ 2^{k}},\frac{a+1}{ 2^{k}}\right)\times
\left(\frac{b}{ 4^{j}},\frac{b+1}{ 4^{j}}\right)\times
\left(\frac{c}{ 2^{k}},\frac{c+1}{ 2^{k}}\right),\]
where $a,b,c\in \N$ and $a,c\le2^k -1$ and $b\le 4^j -1$.
Moreover, $Q_{j,k}(a,b,c)$ is an interior $(j,k)$-brick if 
$
a,b,c >0 \mbox{ and }\frac{a+1}{ 2^{k}}, \frac{b+1}{ 4^{j}}, \frac{c+1}{ 2^{k}} <1.
$
For $k>1$, such bricks map onto bricks: $F(Q_{j,k}(a,b,c)) = Q_{j+1,k-1}(a',b',c'),$ for some $a',b',c'\in\N$.
Therefore, there are $a',b',c'\in\N$ for which
\begin{align*}
F^{k}(Q_{j,k}(a,b,c)) =& Q_{j+k,0}(0,b',0)\\
=& \left(0,1\right)\times\left(\frac{b'}{ 4^{j+k}},\frac{b'+1}{ 4^{j+k}}\right)\times\left(0,1\right),\\
F^{-j}(Q_{j,k}(a,b,c)) =& Q_{0,j+k}(a',0,c')\\
=& \left(\frac{a'}{ 2^{j+k}},\frac{a'+1}{ 2^{j+k}}\right)\times
\left(0,1\right)\times
\left(\frac{c'}{ 2^{j+k}},\frac{c'+1}{ 2^{j+k}}\right).
\end{align*}
It follows that almost every point of $\cube$ is contained in arbitrarily small bricks, and
 hence 2-unstable hyperbolic periodic points of the 3D-baker map are dense.

\subsection{Density of biased points}\label{sect-bias}
In Sec.~\ref{brick-baker} we have seen that almost every point in $[0,1]^3$ 
is contained in an arbitrarily small interior brick, and hence 2-unstable periodic points are dense. We extend this argument to the 3D-HC map.
The new difficulty 
is to determine the length of the box in the $Z$ direction.
To this end we need some preliminary considerations.

For the rest of this subsection, let $F$ be the 3D-HC map.
We say $p\in\cube$ is {\bf regular} if for each $n\in\mathbb Z$, 
$F^n(p)$ is in the interior of a symbol set for $F$. We say a point $p$ is {\bf irrational} 
if each of its coordinates is irrational. 
The following proposition is elementary, but we include a proof for readers' convenience.
\begin{proposition}\label{prop:regular}
Every irrational point in $[0,1]^3$ is regular. In particular, almost every point in $[0,1]^3$ is regular. 
\end{proposition}

\begin{proof}
Any one-dimensional map
$x\in\mathbb R \mapsto c + d x\in\mathbb R$ with $c,d\in\mathbb Q\setminus\{0\}$ 
 maps irrational numbers to irrational numbers.
On each symbol set, $F$ and $F^{-1}$ are linear diagonal maps, and 
each component has this form.
Hence 
they map irrational points to irrational points. 
Since there is no irrational point on the boundaries of symbol sets and
almost every point is irrational,
the assertion follows.
\end{proof}

If $p\in\cube$ is a regular point,
 then 
$\log_2d_Z(F^n(p))=-1$ if $F^n(p)\in R_1$ and 
    $\log_2d_Z(F^n(p))=1$ if $F^n(p)\in R_2$.
    We are interested in an atypical point for which the latter case is predominant.
Let $m,n\in \mathbb Z$ with $m\le n$ and define 
\begin{eqnarray*}
\Phi(m,n;p)=
  \begin{cases}
    0&\mbox{ for }m=n;\\
    \displaystyle\sum_{i=m}^{n-1}\log_2d_Z(F^i(p))&\mbox{ for }m<n.
  \end{cases}
\end{eqnarray*}
We say a regular point $p\in[0,1]^3$ is a 
{\bf biased point}
if
\begin{equation}
\lim_{m\to-\infty}\Phi(m,0;p)=+\infty\quad\mbox{and}\quad
\lim_{n\to+\infty}\Phi(0,n;p)=+\infty.\label{n1}
\end{equation}
Since $\tau$ is ergodic and
the volume of $R_1$ is $\frac{2}{3}$ while that of $R_2$ is $\frac{1}{3}$, typical trajectories have more of the iterates in $R_1$ than in $R_2$. Hence the set of biased points has measure zero. 
Nonetheless, this set is dense in $\cube$.


\begin{proposition}\label{prop:dense-bias}
The set $\{(x,y,z)\in[0,1]^3\colon\text{ biased, }z\notin\mathbb Q\}$ is dense in $[0,1]^3$.
\end{proposition}
\begin{proof}

Fix 
$x_0,y_0\in(0,1)$
 such that the second condition in \eqref{n1} 
 holds for $p=(x_0,0,0)$ and the first condition in \eqref{n1} holds for $p=(0,y_0,0)$. The sets
 $D_{x_0}=\bigcup_{n=0}^\infty\{\tau^{-n}(x_0)\}$ and
 $D_{y_0}=\bigcup_{n=0}^\infty\{\sigma^{-n}(y_0)\}$
 are dense in $[0,1]$.
The set
 $D_{x_0}\times D_{y_0}\times ([0,1]\setminus\mathbb Q)$ is dense in $\cube$ and consists of biased points. 
\end{proof}

Let $p$ be a biased point. We say an integer $k\geq1$ is 
{\bf right biased} for $p$
 if 
\begin{equation}\label{ineq:biased_pair1}
\Phi(i,k;p) > 0\mbox{ for all }i<k.
\end{equation}
For an integer $j\ge1$, 
$-j$ is 
{\bf left biased} for $p$
if
\begin{equation}\label{ineq:biased_pair2}
\Phi(-j,i;p) > 0\mbox{ for all } i > -j.
\end{equation}
We say 
$(j,k)$ is a {\bf biased pair} for $p$ if $j\ge1,k\ge1$ and $-j$ and $k$ are left and right biased for $p$, respectively.

\begin{proposition}\label{prop:infinite-bias}
Let $p\in[0,1]^3$ be a biased point. For any $N>0$
there exist $j>N$ and $k>N$ such that $(j,k)$ is a biased pair for $p$.
\end{proposition}
\begin{proof}
The first condition in \eqref{n1} implies that $\min_{i\leq0}\Phi(i,0,p)\leq0$ exists.
By the second condition in \eqref{n1}, for any $N>0$ there exists $k>N$ such that 
 \begin{equation}\label{bias-new}\Phi(0,k;p)>\max_{0\leq i\leq k-1}\Phi(0,i;p)\quad\text{and}\quad
 \Phi(0,k;p)>-\min_{i\leq0}\Phi(i,0,p).\end{equation}
 The first inequality in \eqref{bias-new} implies $\Phi(i,k;p)=\Phi(0,k;p)-\Phi(0,i;p)>0$ for all $0\leq i\leq k-1$.
 The second inequality in \eqref{bias-new} implies
 $\Phi(i,k;p)=\Phi(0,k;p)+\Phi(i,0;p)>0$
 for all $i\leq0$.
 Therefore
\eqref{ineq:biased_pair1} holds and so
$k$ is right biased for $p$.
The same argument shows that for any $N>0$ there exists $j>N$ such that  $-j$ 
is left biased for $p$.
\end{proof}

\subsection{Construction of bricks for the 3D-HC map}\label{brick-hc}
In this subsection we construct interior $(j,k)$-bricks for the 3D-HC map $F$.
\begin{proposition}[Interior brick]\label{prop:InteriorBrick}
Let $p=(x,y,z)\in[0,1]^3$ be a biased point such that $z\notin\mathbb Q$. For any open subset $U$ of $[0,1]^3$
containing $p$, there exist
a biased pair $(j,k)$ for $p$
and an interior $(j,k)$-brick 
which is contained in $U$.
\end{proposition}
\begin{proof}
Let $j,k\geq1$ be such that $(j,k)$ is a biased pair for $p$.
Define $X^0$ to be the maximal open subinterval of $(0,1)$ containing $x$
on which $\tau^k$ is continuous.
Define $Y^0$ to be the maximal open subinterval of $(0,1)$ containing $y$
on which $\sigma^j$ is continuous.
Since $k$ is right biased for $p$,
$\Phi(0,k;p)\geq1$ holds.
Define $Z^0$ to be the interval containing $z$ of the form
\[Z^0 = \left(\frac{c}{2^{\Phi(0,k;p)}},\frac{c+1}{2^{\Phi(0,k;p)}}\right),\quad
c\in\{0,1,\ldots,2^{\Phi(0,k;p)}-1\}.\] 
For each $-j\leq i\leq k$, set $Q^i=F^i(Q^0)$ and write
$Q^i= X^i\times Y^i\times Z^i.$
For all $-j\leq i\leq k-1$,
$Q^i$ is a subset of some symbol set for $F$.
Moreover,
$Q^{-j}$ is a $Y$ breadbox 
and $Q^{k}$ is a $XZ$ pizzabox.
  Therefore,
$Q^0$ is a $(j,k)$-brick containing $p$. 
We have $|X^0|,|Z^0|\to0$ as $k\to\infty$ and
 $|Y^0|\to0$ as $j\to\infty$.
In what follows,
for sufficiently large $j$ and $k$
we show that
\begin{equation}\label{three}
\overline{X^{-j}}\subset(0,1),   \overline{Y^k}\subset(0,1),  \overline{Z^{-j}}\subset(0,1),
\end{equation}
which implies that $Q^0$ is an interior $(j,k)$-brick, finishing the proof of Prop.\ref{prop:InteriorBrick}.

We show \eqref{three} using a separate argument for each of the three intervals. First,
let $(j,k)$ be a biased pair for $p$
such that the following hold:
\begin{equation}\label{jk}\left(x-\frac{1}{3^k},x+\frac{1}{3^{k}}\right)\subset(0,1),\ \left(y-\left(\frac{2}{3}\right)^j,y+\left(\frac{2}{3}\right)^j\right)\subset(0,1).
\end{equation}
By Prop.~\ref{prop:infinite-bias}, there are infinitely many such biased pairs $(j,k)$ for $p$.
Note that $|X^i|$ decreases as $i$ decreases. 
In particular
$|X^i|=\frac{1}{3^{k-i}}$ holds for all $-j\leq i\leq k$. 
Since $x\in X^0$ and $|X^0|=\frac{1}{3^k}$, the first condition in \eqref{jk} gives $\overline{X^0}\subset(0,1)$.
From the definition of $\tau$ in \eqref{tau}, if $\overline{X^{-j}}$ contains $0$ or $1$, 
then $\overline{X^{1-j}}$ contains $0$ or $1$, and by induction 
$\overline{X^{0}}$ contains $0$ or $1$, a contradiction.
Hence $\overline{X^{-j}}\subset(0,1)$ holds.
The same type of argument works for $Y^i$.
Here $|Y^i|$ decreases as $i$ increases. In particular
$|Y^{i}|\le(\frac{2}{3})^{i+j}$
 holds for all $-j\leq i\leq k$.
 Since $y\in Y^0$, the second condition in \eqref{jk} gives 
 $\overline{Y^0}\subset(0,1)$.
  Hence $\overline{Y^{k}}\subset(0,1)$
 holds.

A proof of $\overline{Z^{-j}}\subset(0,1)$ is more difficult since  
$|Z^i|$ is not monotone in $i$.
Set
\[
\beta(p)=-\min_{i\le 0}\Phi(i,0;p) \ge 0,
\]
and fix $M\leq0$ such that
\begin{equation}\label{M}
\beta(p)=-
\Phi(M,0;p).
\end{equation}
Let $(j,k)$ be a biased pair for $p$ such that
 \begin{equation}\label{j}
    j\geq M-1.
\end{equation}
Since $z$ is irrational and $z\in Z^0$,
$\overline {Z^0}$ does not contain any rational number of the form 
\begin{equation}\label{binary-end}
\frac{a}{2^{\beta(p)}}
\mbox{ for some }a\in\N,
\end{equation}
provided $k$ is large enough. In the next paragraph we show  that if $\overline{Z^{-j}}$ contains $0$ or $1$,
then $\overline{Z^0}$ contains a rational number of the form
in \eqref{binary-end}. This yields a contradiction and hence $\overline{Z^{-j}}\subset(0,1)$ holds.

If $d_Z(F^i(p))=2$ and  
$\overline{Z^{i}}$ contains $0$ or $1$, then
 $\overline{Z^{i+1}}$ contains $0$ or $1$.
If $d_Z(F^i(p))=\frac{1}{2}$
and  
$\overline{Z^{i}}$ contains $0$ or $1$, then $0$ or $\frac{1}{2}$ or $1$ is an endpoint of $\overline{Z^{i+1}}$, and if moreover $d_Z(F^{i+1}(p))=2$, then  $\overline{Z^{i+2}}$ contains $0$ or $1$.
Suppose $\overline{Z^i}$ contains $0$ or $1$. 
As $i$ increases
any number of consecutive ``shrinking'' $i$, \ie, $d_Z=\frac{1}{2}$, followed by an equal or greater number of expansions, \ie, $d_Z=2$, results in an interval whose closure contains $0$ or $1$.
Since $-j<M$ and $\Phi(-j,M,p)\geq0$
by \eqref{M} and \eqref{j},
it follows that
$\overline{Z^M}$ contains $0$ or $1$.
Then, from \eqref{M}
it follows that $\overline{Z^0}$
contains a rational number of the form  in \eqref{binary-end}.
\end{proof}

\subsection{Density of periodic points}\label{dense-per}
The following theorem establishes one part of Thm.~\ref{thm:main},
namely, 
(ii$^\prime$) in the definition of hetero-chaos 
in Sec.~\ref{sec:defheterochaos}.

\begin{theorem}\label{thm:2D_unstable_points} 
Each non-empty open subset of $\cube$ contains a 2-unstable hyperbolic periodic point 
and a 1-unstable hyperbolic periodic point of the 3D-HC map.
\end{theorem}
\begin{proof}
Let $U\subset[0,1]^3$ be an open set.
From Prop.~\ref{prop:dense-bias} and
Prop.~\ref{prop:InteriorBrick},
there exist a biased point $p\in U$, a biased pair $(j,k)$ for $p$
and a $(j,k)$-brick $Q^0$ contained in $U$.
By Prop.~\ref{prop:periodicpoint}, there exists a hyperbolic periodic point in the interior of $Q^0$ that is $2$-unstable for $F$.
Exchanging the roles of $F$ and $F^{-1}$ we obtain
a dense set of hyperbolic periodic points which are $2$-unstable for $F^{-1}$. 
Since these periodic points are $1$-unstable for $F$, the proof is complete.
\end{proof}

\subsection{The existence of a dense trajectory}
\label{sec:transitivity}

 The following theorem establishes 
(iii) in the definition of hetero-chaos 
in Sec.~\ref{sec:defheterochaos}, the remaining part of Thm.~\ref{thm:main}.
\begin{theorem}\label{thm:transitivity}~The 3D-HC map has a trajectory that is dense in $[0,1]^3$.
\end{theorem}

\begin{proof} 
From Prop.~\ref{prop:dense-bias}
and Prop.~\ref{prop:InteriorBrick}, there exist a countable dense set of biased points $\{p_s\}_{s=0}^\infty$ $\subset(0,1)^3$, and 
for each $s\geq0$ a biased pair $(j_s,k_s)$ for $p_s$
and an interior $(j_s,k_s)$-brick $Q^0_s$ which contains $p_s$ and is contained in the open ball of radius $\frac{1}{(s+1)}$ about $p_s$.
We will construct a strictly increasing sequence $\{n_s\}_{s=0}^\infty$ in $\mathbb N$ and a decreasing sequence $\{U_s\}_{s=0}^\infty$ of $Y$ breadboxes
such that $\overline{U_{s+1}}\subset U_s$ and $F^{n_s}(U_s)\subset Q^0_s$ hold for every $s\geq0$.
 Then $\bigcap_{s=0}^\infty U_s=\bigcap_{s=0}^\infty \overline{U_s}$ is the intersection of a nested sequence of compact sets and so is non-empty.
 The forward trajectory of any point in this intersection intersects each of the $Q^0_s$ 
 and so is dense in $[0,1]^3$.

A key ingredient for our construction is the next lemma 
which allows us to connect two bricks.
\begin{lemma}
[Two-bricks lemma]
Let $j\geq0$, $k,m\geq1$ and let $U$ be an
interior $(0,m)$-brick and
let $Q^0$ be an interior $(j,k)$-brick. 
Then $Q=F^{m}(U)\cap F^{-j}(Q^0)$ is an $(m, j+k)$-brick and 
 $F^{-m}(Q)$ is an interior  $(0,j+k+m)$-brick with
$\overline{F^{-m}(Q)}\subset U$.
\label{prop:2brick}
\end{lemma}

We start the construction with $n_0=j_0$ and $U_0=F^{-j_0} (Q^0_1)$. Note that $U_0$ is an interior $(0,n_0+k_0)$-brick.
Given a $(0,n_s+k_s)$-brick $U_s$ for some $s\geq0$, define $n_{s+1}= n_s+j_{s+1}$ and 
$U_{s+1}=F^{-n_s-k_s}\left(F^{n_s+k_s}(U_s)\cap F^{-j_{s+1}}(Q_{s+1}^0)\right).$
By Lem.~\ref{prop:2brick}, $U_{s+1}$ is a $(0,n_{s+1}+k_{s+1})$-brick and satisfies
$\overline{U_{s+1}}\subset U_s$.
It follows from this construction that
$F^{n_s}(U_s)\subset Q^0_s$ holds for every $s\geq0$.

To prove Lem.~\ref{prop:2brick}, set $U'=F^{-m}(Q)$.
 Since $F^{m}(U)$ is an $XZ$ pizzabox and 
 $F^{-j}(Q^0)$ is a $Y$ breadbox, they intersect
each other.
Set $m'=j+k+m$. Notice that $F^{m'}(U')$ is an $XZ$ pizzabox,
and $U'$ is a $Y$ breadbox and satisfies $F^{m+j}(U')\subset Q^0$.
Since $U$ is a $(0,m)$-brick and $U'\subset U$, $F^n(U')$
is contained in a symbol set for $0\leq n\leq m-1$.
Similarly, 
since $Q^0$ is a $(j,k)$-brick
and
$F^{m}(U') \subset F^{-j}(Q^0)$,
$F^n(U')$ is contained in a symbol set for $m\leq n\leq m'-1$.
Hence, $U'$ is a $(0,m')$-brick. 
Since $U$ and $Q^0$ are interior bricks, it follows 
that $U'$ is an interior brick.
Hence 
$\overline{U'}\subset U$. This completes the proof of Lem.~\ref{prop:2brick} and that of Thm.~\ref{thm:transitivity}.
\end{proof}

\section{Ergodicity and further properties}\label{sec:ergodicity}
\ifincludeXX
\begin{figure}
\centering
\includegraphics[width=.75\textwidth]{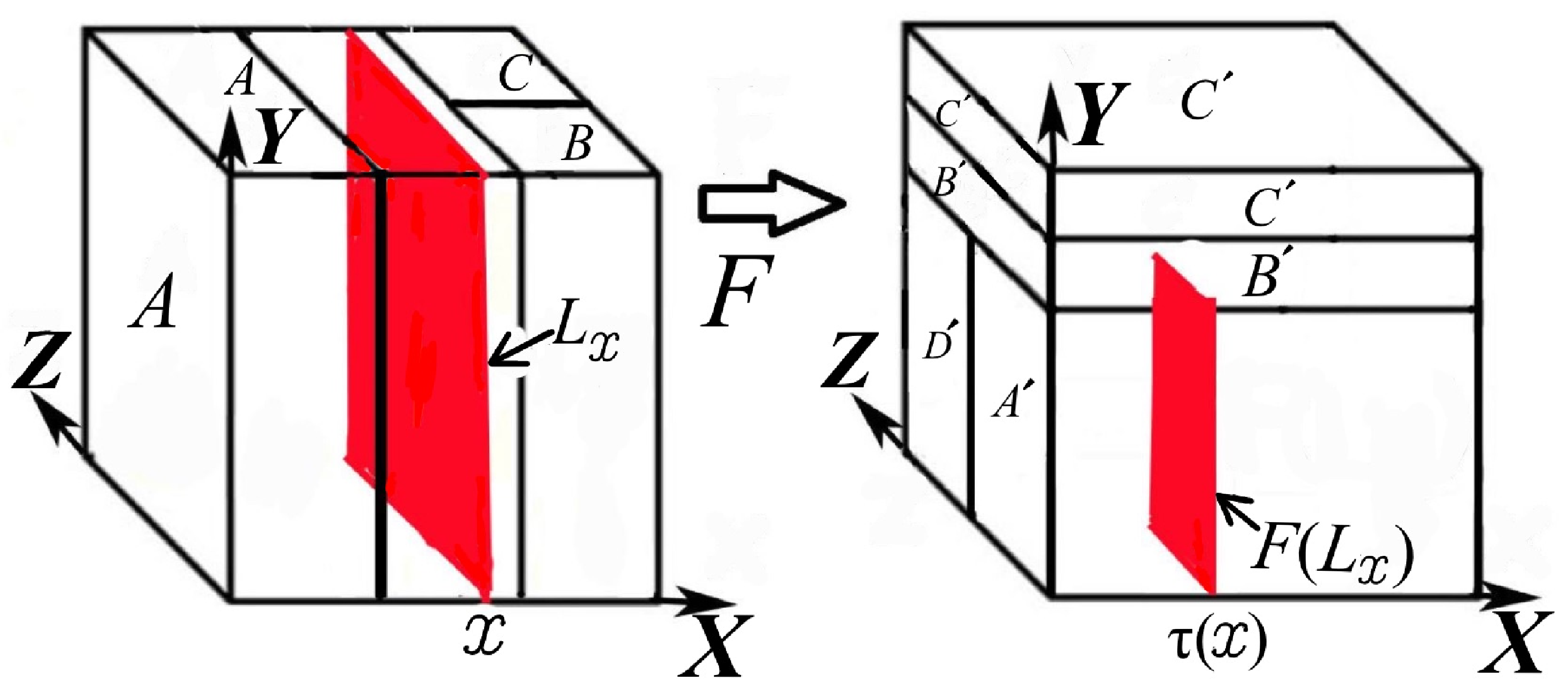}
\caption{
{\bf 
A Leaf and its image.}
The leaf $L_x$, $x\in[0,1]$ is shown on the left and its image  $F(L_x)\subset L_{\tau(x)}$ is shown on the right.}\label{fig:leaf}
\end{figure}
\fi 

This section is organized as follows. In Sec.~\ref{erg-pf} we prove Thm.~\ref{thm:ergodic0}.
In Sec.~\ref{LY} we prove
Thm.~\ref{li-yorke}.
In Sec.~\ref{LN} we compute Lyapunov numbers of typical points.

\subsection{Ergodicity}\label{erg-pf}
For each $x\in[0,1]$ we set
 \[L_x=\{x\}\times[0,1]^2,\]
 and call $L_x$ a {\bf leaf} through $x$, or 
 simply a leaf. See Fig.~\ref{fig:leaf}. Since
 the $Y$ direction is uniformly contracting everywhere for the 3D-HC map,
 the forward iterations of leaves are contained in thin strips.
The next proposition asserts that diameters of almost all leaves shrink to $0$ under forward iteration.

\begin{proposition}
[Asymptotic contraction on almost all leaves]
\label{prop:leaf}
Let $F$ be the 3D-HC map.
 For a.e. $x\in[0,1]$
and every $n\geq0$, there exist intervals $Y_n$, $Z_n$ such that $F^n(L_x)=\{\tau^n(x)\}\times Y_n\times Z_n$, and $|Y_n|\to0$, $|Z_n|\to0$ as $n\to\infty$.
\end{proposition}


\begin{proof}
For $x_0\in [0,1]$ and $n\geq0$
write 
$p_n=F^n((x_0,0,0))$.
Set $Y_0=[0,1]$ and $Z_0=[0,1]$.
By induction one can show that
for every $n\geq0$ there exist intervals $Y_n$, $Z_n$ such that $F^n(L_x)=\{\tau^n(x)\}\times Y_n\times Z_n$, and the following hold:

\begin{itemize}

\item[$\circ$] $|Z_n|=1$ or $|Z_n|\leq 1/2;$

\item[$\circ$] If $p_n\in R_1$, then $|Z_{n+1}|=|Z_n|/2;$

\item[$\circ$] If $p_n\in R_2$ and $|Z_n|=1$, then  $|Z_{n+1}|=|Z_{n}|=1$;

\item[$\circ$] If $p_n\in R_2$
and $|Z_n|\leq 1/2$, then
 $|Z_{n+1}|= 2|Z_n|$.
\end{itemize}
Since $F$ contracts $|Y_n|$ by a factor of $\frac{2}{3}$ or $\frac{1}{6}$,
$|Y_n|\to0$ as $n\to\infty$. 
Write $x_0$ as a base 3 number and denote by $b_k\in\{0,1,2\}$
 the $k$-th trinary digit. Then 
$p_{k-1}\in R_1$ for $b_k=0,1$ and 
$p_{k-1}\in R_2$ for $b_k=2$. This yields
\[|Z_{n}|\leq2^{-\#\{1\leq k\leq n\colon b_{k}\in \{0,1\}\}}\cdot 2^{\#\{1\leq k\leq n\colon b_k=2\}}\] for all $n\geq1$.
Since $\tau$ is ergodic,
 the digits $0,1,2$ appear equally likely for almost every $x_0$, and therefore $|Z_n|\to0$ as $n\to\infty$.
\end{proof}

\begin{proof}[Proof of Thm.~\ref{thm:ergodic0}]
Let $F$ be the 3D-HC map.
Let $\phi: \cube\to \mm R$ be a continuous function. 
Since $\tau$ is ergodic,
by Birkhoff's ergodic theorem,
$\frac{1}{n}\sum_{i=0}^{n-1}\phi((\tau^i(x),0,0))$ converges a.e. to a constant as $n\to\infty$.
Since $F^i((x,0,0))=(\tau^i(x),0,0)$,
this limit is equal to the trajectory average $\avphi((x,0,0))$ in \eqref{average}. 
To show that $\avphi$ is constant a.e.,
it is enough to show that
 for almost every $x\in[0,1]$,
$ \avphi$ is constant on the leaf $L_x$.
This is a consequence of Prop.~\ref{prop:leaf} and the uniform continuity of $\phi$.
Since $\phi$ is an arbitrary continuous function, the ergodicity of $F$ follows.

Let $F$ be the 2D-HC map and
let $\phi\colon[0,1]^2\to\mathbb R$ be an arbitrary continuous function.
Define $\psi\colon[0,1]^3\to\mathbb R$ by
$\psi(x,y,z)=\phi(x,z)$. 
From the ergodicity of the 3D-HC map, 
there exist a set $\Omega\subset[0,1]^3$ of full measure
and a constant $c\in\mathbb R$ such that
$\langle\psi\rangle(p)=c$
holds for all $p\in \Omega$.
Then $\avphi(\pi(p))=\langle\psi\rangle(p)=c$ holds for all $p\in \Omega$. 
Since $\pi(\Omega)$ is a set of full measure,
$\avphi$ is constant a.e.
Hence $F$ is ergodic. 
\end{proof}

\begin{corollary}
\label{lem:ergodic-transitive}
Let $F$ be the 2D-HC or the 3D-HC map.
For almost every point $p$ in the domain of definition of $F$,
$\{F^n(p)\}_{n\in\mathbb N}$ is dense in the domain.
\end{corollary}
\begin{proof}
This is a consequence of Thm.\ref{thm:ergodic0} and Birkhoff's ergodic theorem.
\end{proof}

\subsection{Almost every pair is a scrambled pair}\label{LY}
Thm~\ref{li-yorke} is a 
consequence of the ideas developed in Sec.~\ref{erg-pf}.
\begin{proof}[Proof of Thm.~\ref{li-yorke}]

The proof of Thm.~\ref{thm:ergodic0} used  leaves in $\cube$ 
and the ergodicity of $\tau$.
Here, we use leaves in $\cube\times\cube$ through $(x_1,x_2)$ of the form
$L_{x_1}\times L_{x_2}$,
and the ergodicity of  
  $\tau\times\tau$ relative to the Lebesgue measure on the product space.
Then, the conclusion of Thm.~\ref{li-yorke} follows from Birkhoff's ergodic theorem.
\end{proof}

\subsection{Lyapunov numbers}\label{LN}
The Lyapunov numbers of a point $p\in[0,1]^3$ are the geometric means of the diagonal elements of
the Jacobian matrix of $F$
along the trajectory. The Lyapunov number $\lambda_X$ in the $X$ direction is $3$,
and that 
in the $Y$ and $Z$ directions are (when they exist)
\begin{equation*}
 \lim_{n\to\infty} \Big(\prod_{i=0}^{n-1}d_Y(F^i(p))\Big)^{1/n}\mbox{ and }
 \lim_{n\to\infty} \Big(\prod_{i=0}^{n-1}d_Z(F^i(p))\Big)^{1/n}.
\end{equation*}
From the ergodicity of $\tau$ and the form of the map $F$,
 typical trajectories spend 
 $\frac{2}{3}$ of its iterates in $R_1$ and $\frac{1}{3}$ in $R_2$. 
The Lyapunov numbers 
of typical points are as follows: 
\begin{align*}
\lambda_X &= 
3;\\ 
\lambda_Y &=
\left(\frac{2}{3}\right)^{2/3}\times\left(\frac{1}{6}\right)^{1/3}=\left(\frac{2}{27}\right)^{1/3}\sim0.42;\\
\lambda_Z&=
\left(\frac{1}{2}\right)^{2/3}\times2^{1/3}=\left(\frac{1}{2}\right)^{1/3}\sim0.79.
\end{align*}
Since $F$ is volume preserving, the product of the three numbers is $1$. 

\subsection*{Acknowledgments} We thank the referees for very useful comments, 
and Suddhasattwa Das, Shin Kiriki, Naoya Sumi for fruitful discussions.
YS was supported by the JST PRESTO JPMJPR16E5 and the JSPS KAKENHI Grant No.17K05360 and 19KK0067. 
HT was supported by the JSPS KAKENHI Grant No.19KK0067, 19K21835 
and 20H01811.


\end{document}